\documentclass{amsart}
\usepackage{amsthm}
\usepackage{amsfonts}
\usepackage{color}

\newtheorem{lemma}{Lemma}
\newtheorem{theorem}{Theorem}
\theoremstyle{definition}
\newtheorem{notation}{Notation}
\newtheorem{remark}{Remark}

\def\bigO{\mathcal{O}}

\def\mn{  \left[
    \begin{array}{c}
      s\\
      m
    \end{array}      
    \right]
}
\def\mna#1#2{  \left[
    \begin{array}{c}
      #1\\
      #2
    \end{array}      
    \right]
}

\def\deru#1{\if1#1
\frac{\partial      u(x)}{\partial x}
\else
\frac{\partial^{#1}      u(x)}{\partial x^{#1}}
\fi
}
\title{The maximal order of semidiscrete schemes for quasilinear first
order partial differential equations}

\author{A.~Baeza, P.~Mulet, D.~Zor\'{\i}o}
\thanks{This work has been financially supported by Spanish MINECO
  projects MTM2011-22741 and MTM2014-54388-P}
\address{antonio.baeza@uv.es,
mulet@uv.es,
david.zorio@uv.es,
Departament de Matem\`atica Aplicada, Universitat de
  Val\`encia, Spain}

\begin{document}
\maketitle
\begin{abstract}
 We prove that  a semidiscrete
  $(2r+1)$-point scheme 
  for quasilinear first order PDE cannot attain an order higher than
  $2r$. Moreover, if the forward 
  Euler fully discrete scheme obtained from the  linearization about
  any constant state of the semidiscrete  scheme is 
  stable, then the upper bound for the order of the scheme is
  $2r-1$. This bound is attained for a wide range of schemes and
  equations.  
\end{abstract}

\section{Introduction}
A basic strategy for obtaining high order numerical methods for
quasilinear first order partial differential equations $u_t=a(u) u_x$
consists in the \textit{method of lines}: the spatial term $a(u) u_x$
is approximated by means of high order finite differences 
\[
\big(a(u)u_x\big)(x_j, t)\approx \frac{H(v_{j-r}(t),\dots,v_{j+r}(t))}{h},
\]
for approximations $v_j(t)\approx u(x_j, t)$, where $h$ is the grid
step size and $x_j=x_0+jh$. The resulting ODE
\begin{equation}\label{eq:11}
v'_{j}(t)=\frac{H(v_{j-r}(t),\dots,v_{j+r}(t))}{h}
\end{equation}
is then solved by some ODE solver and the  fully discrete numerical
method thus obtained 
has an order which is the minimum of the orders of the finite
difference formula and the ODE solver.
Strong Stability Preserving
Runge-Kutta (SSPRK) ODE solvers (see \cite{GottliebShuTadmor2001} and
references 
therein) are widely used for the solution of \eqref{eq:11} due to its
nonlinear stability features. Since 
 the forward Euler scheme is the basic building block for SSPRK
 solvers, the stability of that scheme is crucial for the stability of
 high order  fully discrete  schemes obtained from the method of lines.

The problem of the maximal order that can be attained by a
semidiscrete scheme for first order partial differential equations is
considered  in \cite{JeltschStrack85} for linear schemes for the
linear advection equation. The authors of this work 
conclude that the maximal order of  linear schemes for which
\eqref{eq:11} is stable is $2r$. In \cite{Chin15} 
the author derives conditions for linear, explicit time-marching methods 
approximating the $m$-th order linear equation with constant coefficients
to any order.

In this work we deal with the maximal order that can be attained with
a general semidiscrete scheme for the first-order quasilinear
 partial differential equation. We
prove that the order of any such scheme is bounded above by
$2r$ and that if the order attains this bound then the forward Euler
scheme is unconditionally unstable. We also point out that finite
difference $(2r+1)$-point WENO schemes \cite{LOC94,JS96,ABBM11}
obtained from maximal order reconstructions do attain order $2r-1$.

\section{The maximal order of a semidiscrete scheme}
We start with two lemmas 
that establish the invertibility of some
Vandermonde-like matrices that appear in the analysis of the schemes. 
We then introduce some notation and establish 
a generalization of the chain rule known as Fa\`a di Bruno's formula
\cite{faadibruno1857} 
in Theorem \ref{theorem:fdb}, whose proof is included in an appendix, and we finally state our main result  
in Theorem \ref{theorem:1}.

\begin{lemma}\label{lemma:1}  
  Given $a_1,\dots,a_n$, then
  \[
  \det (a_{j}^{i})_{i,j=1}^{n}=\prod_{j=1}^{n}\prod_{k=j+1}^{n} (a_k-a_j)
  \prod_{j=1}^{n} a_j.
  \]
  If $a_1,\dots,a_n$ are pairwise distinct and not null
  then $\det (a_{j}^{i})_{i,j=1}^{n} \neq 0$.
\end{lemma}
\begin{proof}
  This is easily obtained from the fact that  the
  determinant of a Vandermonde matrix 
  $(a_j^{i-1})_{i,j=1}^{n}$ is given by
  $\det (a_{j}^{i-1})_{i,j=1}^{n}=\prod_{j=1}^{n}\prod_{k=j+1}^{n}
  (a_k-a_j)$.

\end{proof}
\begin{lemma}\label{lemma:2}  
   The determinant of the
  matrix given by the entries
  \[
  A(a_1,\dots,a_n)_{i,j}=
  \begin{cases}
    1&\text{if } i = 1\\
    a_j^{i}&\text{if } i \geq  2
  \end{cases},
  \]
  for any $a_1,\dots,a_n$, $n\geq 2$,
  is given by:
  \[
  \det A(a_1,\dots,a_n)=\prod_{j=1}^{n}\prod_{k=j+1}^{n} (a_k-a_j)
  \sum_{j=1}^{n} \prod_{k=1,k\neq j}^{n} a_k.
  \]
  If $a_1=0,a_2,\dots,a_n$ are pairwise distinct 
  then $\det  A(a_1,\dots,a_n) \neq 0$.
\end{lemma}
\begin{proof}
  The result is proved by induction on $n$. The result for $n=2$ is
  \begin{equation*}
    \det
    \begin{bmatrix}
      1&1\\
      a_1^2&a_2^2
    \end{bmatrix}
    =(a_2-a_1)(a_1+a_2),
  \end{equation*}
  which  clearly holds true. Assume now $n>2$ and the result to be
  true for $n-1$. We do 
  elimination in the first column, by  subtracting from row $i+1$  the
  first row multiplied by
  $a_{1}^{i+1}$, $i=1,\dots,n-1$. This   yields
  \begin{equation}\label{eq:460}
    \begin{split}
  \det
  A(a_1,\dots,a_n)&=\det(a_{j+1}^{i+1}-a_1^{i+1})_{i,j=1}^{n-1}
  =\det\left(\left(a_{j+1}-a_1\right)
    \left(\sum_{k=0}^{i}a_{j+1}^{k}a_1^{i-k}\right)\right)_{i,j=1}^{n-1}\\
  &= \prod _{j=2}^{n}(a_{j}-a_1)
  \det\left(\sum_{k=0}^{i}a_{j+1}^{k}a_1^{i-k}\right)_{i,j=1}^{n-1}.
  \end{split}
\end{equation}
Subtracting to row $i$ of the  matrix
\begin{equation*}
  \left(\sum_{k=0}^{i}a_{j+1}^{k}a_1^{i-k}\right)_{i,j=1}^{n-1}
\end{equation*}  
its row $i-1$ multiplied by $a_1$, since
\begin{equation*}
  \sum_{k=0}^{i}a_{j+1}^{k}a_1^{i-k}-a_1\sum_{k=0}^{i-1}a_{j+1}^{k}a_1^{i-1-k}=
  \sum_{k=0}^{i}a_{j+1}^{k}a_1^{i-k}-\sum_{k=0}^{i-1}a_{j+1}^{k}a_1^{i-k}
  =a_{j+1}^i,
\end{equation*}  
we get
$$\det\left(\sum_{k=0}^{i}a_{j+1}^{k}a_1^{i-k}\right)_{i,j=1}^{n-1}
=\det(b_{i,j})_{i,j=1}^{n-1},$$
where
$$b_{i,j}=\begin{cases}
    a_1+a_{j+1} & \mbox{if } i=1\\
    a_{j+1}^i & \mbox{if } 2\leq i\leq n-1.
    \end{cases}$$
Therefore, we can split this determinant into the sum of two
determinants:
$$\det(b_{i,j})_{i,j=1}^{n-1}=a_1\det(A(a_2,\ldots,a_{n-1}))
+\det((a_{j+1}^i))_{i,j=1}^{n-1}.$$
Now, using \eqref{eq:460} and the induction hypothesis for the
determinant of the first 
summand and applying Lemma \ref{lemma:1} for the determinant of the
second one:
\begin{equation*}
  \begin{split}
    \det A(a_1,\dots,a_n)=&\prod _{j=2}^{n}(a_{j}-a_1)
  \left[a_1\prod_{j=2}^{n}\prod_{k=j+1}^{n} (a_k-a_j)
  \sum_{j=2}^{n} \prod_{k=2,k\neq j}^{n} a_k\right.\\
  +&\left.\prod_{j=2}^{n}\prod_{k=j+1}^{n}(a_k-a_j)\prod_{j=2}^{n}
  a_j\right]=\prod_{j=1}^{n}\prod_{k=j+1}^{n} (a_k-a_j)
  \sum_{j=1}^{n} \prod_{k=1,k\neq j}^{n} a_k.
  \end{split}
\end{equation*}
\normalsize
Finally, if $a_1=0,a_2,\dots,a_n$ are pairwise distinct, then:
\[
\det A(a_1,\dots,a_n)
=
\prod_{j=2}^{n}\prod_{k=j+1}^{n} (a_k-a_j)
\prod_{j=2}^{n} a_j^2\neq 0.
\]

\end{proof}  

\begin{notation}
Denote by $\mathcal{M}(s, n)$ the vector space of multilinear
functions ($s$-order --covariant-- tensors) 
\begin{equation*}
  T\colon \left(\mathbb R^{n}\right)^{s}
  \to \mathbb R, \quad \left(\mathbb R^{n}\right)^{s}=\overbrace{\mathbb{R}^n\times\cdots\times\mathbb{R}^n}^{s}.
\end{equation*}
Since $\left(\mathbb R^{n}\right)^{s}$ is
isomorphic to the vector space of $n\times s$ matrices, we can regard
$s$-order tensors as acting on the columns of $n\times s$ matrices. Tensors
can be characterized as $\overbrace{n\times \dots\times n}^{s}$ matrices
$(T_{i_1,\dots,i_s})$, i.e., their action on an $n\times s$ matrix
$A$ is given by
\begin{equation*}
  T(A)=\sum_{i_1=\dots=i_s=1}^{n} T_{i_1,\dots,i_s} A_{i_1,1}\dots
  A_{i_s,s}.
\end{equation*}  

Assume $T\colon \mathbb R^n \to \mathcal{M}(s, n)$ is differentiable
(equivalently, $T_{i_{1},\dots,i_{i_s}}$ are differentiable). Then
$T'(u)\in \mathcal{M}(s+1, n)$ is given by:
\begin{equation*}
  T'(u)_{i_0,i_1,\dots,i_s}=\frac{\partial T_{i_1,\dots,i_s}}{\partial
  u_{i_0}}(u).
\end{equation*}

With this notation  the derivatives of real functions can be mapped
to tensors in the following way:
If $f\colon\mathbb R^n\to \mathbb R$ is in $\mathcal{C}^s$, then we
define $f^{(k)}\in\mathcal{M}(k, n)$, $1\leq k\leq s$ by
\begin{equation*}
  f^{(k)}_{i_1,\dots,i_k}(u)=\frac{\partial^k }{\partial
    u_{i_1}\dots\partial u_{i_k}}f(u).
\end{equation*}
\end{notation}

\begin{notation}
For $s\in\mathbb N$, we denote
\begin{align*}
  \mathcal{P}_{s}&=\{ m\in\mathbb N^{s} /
\sum_{j=1}^{s} j m_j=s \},
\end{align*}  
and  for $m\in\mathbb N^{s}$,
$|m|=\sum_{j=1}^{s} m_j$. For $m\in\mathcal{P}_{s}$, we denote
\begin{equation*}
\mn    =\frac{s!}{m_1!\dots m_{s}!}
  \end{equation*}
  and for $m\in\mathbb N^s$ and a function $u\colon \mathbb
  R\to\mathbb R^{n}$, the $n\times|m|$ matrix $D^m u(x)$ is given
  column-wise by 
  \begin{equation}\label{eq:Dm}
    \begin{aligned}
    (D^m u(x))_{\sum\limits_{l<j}m_l+k}&=\frac{1}{j!}\frac{\partial^{j}
      u(x)}{\partial x^j}, \quad k=1,\dots,m_j,\quad j=1,\dots,s,\\
    D^m u(x)&=\begin{bmatrix}
      \overbrace{\frac{\deru1}{1!} \dots \frac{\deru1}{1!}}^{m_1} &
      \dots
      \overbrace{\frac{\deru{s}}{s!} \dots \frac{\deru{s}}{s!}}^{m_s}
    \end{bmatrix},
  \end{aligned}    
\end{equation}    
where it is to be understood here that $j$-th order derivatives do not
appear in $D^{m}u$ if $m_j=0$.
\end{notation}

We state the generalized chain rule for high order derivatives of
compositions of functions due to Fa\`a di Bruno
\cite{faadibruno1857}. 
  
\begin{theorem}[Fa\`a di Bruno's formula]\label{theorem:fdb}
Let $f\colon \mathbb R^n\to R$, $u\colon \mathbb R\to \mathbb R^n$ be
$s$ times continuously differentiable. Then
\begin{equation}\label{eq:fdb}
  \frac{d^sf(u(x))}{dx^s} =
  \sum_{m\in\mathcal{P}_s} \mn f^{(|m|)}(u(x)) D^m u(x)
\end{equation}
\end{theorem}

\begin{theorem} \label{theorem:1}
Consider the   semidiscrete scheme \eqref{eq:11},
with $H$ a smooth function, 
for the approximate solution of $u_t=a(u)u_x$, with $v_j(t)\approx
u(x_j, t)$. Then the order $p$ of \eqref{eq:11} satisfies  $p\leq 2r$ and if 
$p=2r$ then the forward Euler scheme applied to the linearized
scheme about any constant state is unconditionally unstable.
\end{theorem}

\begin{proof}
We drop the dependence on $t$ until it is required. By the definition of the order of \eqref{eq:11}, for any smooth function $u$
and any $x$ we get:
\begin{equation}
  \label{eq:1}
  \frac{H(u(x-rh),\dots,u(x+rh))}{h}=a(u(x)) u'(x)+\bigO(h^p).
\end{equation}
Let us fix $u, x$ and  denote $u_j(h)=u(x+jh)$, $U(h)=(u_{-r}(h),
\dots,u_r(h))$ and $\Phi(h)=H(U(h))$. 
Then  \eqref{eq:1} and the Taylor development of $\Phi$ about $0$  yield:
\begin{equation}\label{eq:0}
\Phi^{(s)}(0)=\Phi_{u,x}^{(s)}(0)=\delta_{s,1}a(u(x))u'(x),\quad
\forall s=0,\dots,p,\quad \forall u, x.
\end{equation}
Fa\`a di Bruno's formula \eqref{eq:fdb} yields
\begin{equation}\label{eq:fdbH}
\Phi^{(s)}(0)= \sum_{m\in\mathcal{P}_s}
\mn H^{(|m|)}(U(0))D^mU(0),\quad U=U_{u,x},
\end{equation}
for any $u$ and $x$ 

For fixed $s\in\mathbb N$ and $v\in\mathbb R$, we consider in  \eqref{eq:fdbH} $x=0$,
$u(y)=v+\frac{y^{s}}{s!}$, which verifies $u^{(l)}(0)=0$,
$l=1,\dots,s-1$, $u^{(s)}(0)=1$.   Then, for
$m\in\mathcal{P}_s$ such that $m_s=0$, one has $D^{m}U(0)=0$, since the
columns of this matrix are $0=u^{(j)}(0)/j!$ for some $j\in\{1,\dots, s-1\}$. On the other
hand, for $m=(0,\dots,0,1)\in\mathbb{N}^s$, $D^m U(0)$ is a one column
matrix given by
\begin{equation*}
  D^{(0,\dots,0,1)}U(0)=\frac{1}{s!}U^{(s)}(0)
  =\frac{1}{s!}
  \begin{bmatrix}
    u^{(s)}_{-r}(0)\\
    \vdots\\
    u^{(s)}_{r}(0)
  \end{bmatrix}
  =
  \frac{1}{s!}u^{(s)}(x)
  \begin{bmatrix}
    (-r)^s\\
    \vdots\\
    r^s
  \end{bmatrix}    
  =
  \frac{1}{s!}
  \begin{bmatrix}
    (-r)^s\\
    \vdots\\
    r^s
  \end{bmatrix}    
\end{equation*}

Therefore, 
\begin{equation}\label{eq:444}
  \Phi^{(s)}(0)= \mna{s}{e_s} H^{(|e_s|)}(U(0))D^{e_s}U(0)
  =
  \sum_{l=-r}^{r} \frac{\partial H}{\partial
    u_{l}}(v,\dots,v)l^s. 
\end{equation}
From \eqref{eq:444} for $s=1$ and \eqref{eq:0} applied to $u(y)=v+y$
we obtain
\[
\Phi'(0)=\sum_{l=-r}^{r}l\frac{\partial H}{\partial
  u_{l}}(v,\dots,v)=a(v),
\]
for any $v\in\mathbb R$.

On the other hand, equation \eqref{eq:0} for $s=0$  reads as $0=\Phi(0)=H(v,\dots,
v)$ for any $v$$(=u(x))$, which in turn yields that 
\begin{equation}\label{eq:447}
  \sum_{l=-r}^{r}\frac{\partial H}{\partial u_l}(v, \dots, v)=0,
\end{equation}
for any $v$.

Recapping, we have obtained the following equations:
\begin{equation}\label{eq:446}
  \delta_{k,1}a(v)=\sum_{l=-r}^{r}l^k\frac{\partial H}{\partial u_l}(v,\dots,v),\quad
  k=0,\dots,p, \quad\forall v.
\end{equation}
Now, if $p>2r$, then
\begin{equation}\label{eq:445}
  0=\sum_{l=-r}^{r}l^k\frac{\partial H}{\partial u_l}(v,\dots,v),\quad
  k=0,\dots,2r+1, k\neq 1,
\end{equation}
i.e., $\big(\frac{\partial H}{\partial u_{-r}}(v,\dots,v),\dots
\frac{\partial H}{\partial u_{r}}(v,\dots,v)\big)$ is a solution of
system \eqref{eq:445}, whose matrix, with the notation of Lemma
\ref{lemma:2}, is  $A(-r, \dots, 0, \dots, r)$ and satisfies
\begin{equation*}
  \det A(-r, \dots,
0, \dots, r)\neq 0,
\end{equation*}
so the unique
solution of system \eqref{eq:445} is the trivial one, i.e.,
\begin{equation*}
  \frac{\partial H}{\partial u_l}(v,\dots,v)=0, \quad
l=-r,\dots,r, \forall v,
\end{equation*}
which contradicts  the fact that, from \eqref{eq:446} for $s=1$,
$\sum_{l=-r}^{r}l\frac{\partial H}{\partial u_l}(v,\dots,v)=a(v)$  for
generic 
$a$ and $v$. Therefore $p\leq 2r$.

If $p=2r$ then
\[
\sum_{l=1}^{r}l^{2k}(\frac{\partial H}{\partial u_l}(v,\dots,v)+\frac{\partial H}{\partial u_{-l}}(v,\dots,v))=0,\quad k=1,\dots,r,
\]
and this, after Lemma \ref{lemma:1}, gives
\begin{equation}\label{eq:470}
\frac{\partial H}{\partial u_l}(v,\dots,v)+\frac{\partial H}{\partial
  u_{-l}}(v,\dots,v)=0, \quad l=1,\dots,r.
\end{equation}
On the other hand, since
$\sum_{l=-r}^{r} \frac{\partial H}{\partial u_l}(v,\dots,v)=0$ by
\eqref{eq:447}, we get from \eqref{eq:470}
\begin{equation}\label{eq:461}
\frac{\partial H}{\partial u_0}(v,\dots,v)=0.
\end{equation}

If $\bar v$ is any constant, since $H(\bar v, \dots, \bar v)=0$, then
the linearized scheme for $v_j(t)=\bar v+w_j(t)$ reads as:
\begin{equation*}
w'_j=\frac{\frac{\partial H}{\partial_{u_{-r}}}(v,\dots, \bar
  v)w_{j-r}+\dots+\frac{\partial H}{\partial_{u_{r}}}(v,\dots,
  v)w_{j+r}}{h},
\end{equation*}
which, by \eqref{eq:460} and \eqref{eq:461} can be written as
\begin{equation}\label{eq:448}
w'_j=\frac{\sum_{l=1}^{r}\alpha_{l}(w_{j+l}-w_{j-l})}{h},
\end{equation}
where $\alpha_l=\frac{\partial H}{\partial u_l}(\bar v,\dots, \bar   v)$.  
The forward Euler scheme applied to the  system of ODE \eqref{eq:448} is:
\[
w_j^{n+1}=w_j^{n}+\frac{k}{h}\sum_{l=1}^{r}\alpha_l(w^n_{j+l}-w^n_{j-l}),
  \]
or, in linear operator form, $w^{n+1}=\Psi_{k/h} w^n$. The   Fourier
transform of $\Psi_{k/h}$  is: 
\[
\widehat{\Psi}_{k/h}(\theta)=1+2i\frac{k}{h}\sum_{l=1}^{r}\alpha_l
\sin(\theta l), i=\sqrt{-1},
\]
which satisfies $|\widehat{\Psi}_{k/h}(\theta)|>1$ for any $k/h$ and
some $\theta$, that
is, the forward Euler scheme for the system of ODE \eqref{eq:448} is
unstable for any $k/h$.

\end{proof}

\begin{remark} The Lax-Wendroff scheme is a second order 3-point
scheme for conservation laws. Nevertheless, this  is  not a
contradiction to  Theorem \ref{theorem:1}, for this scheme is fully discrete. On the other
hand,  $(2r+1)$-point 
smooth and stable schemes of order $2r-1$ for conservation laws can be
obtained, for example, via finite difference WENO schemes with global Lax-Friedrichs
flux splittings \cite{Shu09,ABBM11}. 
\end{remark}

\appendix

\section{}

For the sake of completeness, we include in this appendix a proof of
Theorem \ref{theorem:fdb}, for we have not found satisfactory references
for its proof.

The following result is easily established.
\begin{lemma}\label{lemma:3}
Assume $T\colon \mathbb R^n \to \mathcal{M}(s, n)$ is differentiable
(equivalently, $T_{i_{1},\dots,i_{i_s}}$ are differentiable) and that
$A\colon \mathbb R\to \mathbb R^{n\times s}$,
$u:\mathbb{R}\rightarrow\mathbb{R}^n$ are also differentiable.
Then, $\forall x\in\mathbb{R}$
\begin{equation*}
  \frac{d}{dx}T(u(x))A(x)=T'(u(x))[u'(x)\ A(x)]+T(u(x))\sum_{j=1}^{s} d_j A(x),
\end{equation*}  
where we have used the notation
$d_j A(x)$ for the $n\times s$ matrix given by the columns:
\begin{equation*}
  (d_j A(x))_{k}=\begin{cases}
    A_{k}(x) & k\neq j\\
    A'_{j}(x) & k= j
  \end{cases}    
\end{equation*}
\end{lemma}

\begin{notation}
  We introduce some further notation for the proof of Theorem
  \ref{theorem:fdb}. 
For $s\in\mathbb N$, we denote
\begin{align*}
  \mathcal{P}_{s,j}&=\{ m\in\mathcal{P}_s /
m_j\neq 0 \}.
\end{align*}  
We denote also
\begin{equation*}
  S_0\colon \mathcal{P}_{s}\to \mathcal{P}_{s+1,1},\quad
  S_0(m)_k=
  \begin{cases}
    0 & k=s+1\\
    m_k& s\geq k\neq 1\\
    m_{1}+1 & k=1,
  \end{cases}    
\end{equation*}  
\begin{equation*}
  S_j\colon \mathcal{P}_{s,j}\to \mathcal{P}_{s+1,j+1},\quad
  S_j(m)_k=
  \begin{cases}
    0 & k=s+1\\
    m_k& s\geq k\neq j, j+1\\
    m_{j}-1 &s\geq  k=j\\
    m_{j+1}+1 &s\geq  k=j+1.
  \end{cases}    
\end{equation*}
for $1\leq j<s$,  and $S_s$ that maps $(0,\dots,0,1)\in\mathbb N^s$ to $(0,\dots,0,1)\in\mathbb N^{s+1}$.
\end{notation}

\begin{proof}{(of Theorem \ref{theorem:fdb})}
We use induction on $s$, the case $s=1$ being the chain rule. By the
induction hypothesis for $s$ and Lemma
\ref{lemma:3} we deduce:
\begin{gather*}
  \frac{d^{s+1}f(u(x))}{dx^{s+1}} =
  \sum_{m\in\mathcal{P}_s} \mn \frac{d}{dx}\left(f^{(|m|)}(u(x)) D^m u(x)\right)\\
  =\sum_{m\in\mathcal{P}_s} \mn \big((f^{(|m|)})'(u(x)) [u'(x)\ D^m u(x)]+
  f^{(|m|)}(u(x)) \sum_{j=1}^{n}d_j D^m u(x)\big)\\
  =\sum_{m\in\mathcal{P}_s} \mn \big(f^{(|m|+1)}(u(x)) [u'(x)\ D^m u(x)]+
  f^{(|m|)}(u(x)) \sum_{j=1}^{n}d_jD^m u(x)\big).
\end{gather*}
Now,
\begin{equation*}
  d_j D^{m} u(x)=D^{S_j(m)} u(x) P E,
\end{equation*}  
where $P$ is a permutation matrix corresponding to the transposition of
$j$ and $\sum_{l\leq k} m_l$, with $\sum_{l< k} m_l < j \leq
\sum_{l\leq k} m_l$ and $E$ is a diagonal matrix with $k+1$
in the $\sum_{l\leq k} m_l$ entry and 1 in the rest.

 By the symmetry of $f^{(|m|)}$, if  $\sum_{l< k} m_l < j \leq
\sum_{l\leq k} m_l$
\begin{equation*}
  f^{(|m|)}(u(x)) d_jD^m u(x)=
  (k+1)
  f^{(|S_{k}(m)|}(u(x)) D^{S_{k}(m)} u(x),
\end{equation*}
therefore, collecting identical terms,
\begin{multline*}
    \frac{d^{s+1}f(u(x))}{dx^{s+1}} =
    \sum_{m\in\mathcal{P}_s} \mn \big(f^{(|S_0 (m)|)}(u(x)) D^{S_0
                                      (m)}u(x)\\
  +
    \sum_{j=1}^{n}f^{(|m|)}(u(x)) d_jD^{S_j (m)} u(x)\big)
  \end{multline*}
  can be written as
\begin{equation}\label{eq:4}
  \begin{aligned}
        \frac{d^{s+1}f(u(x))}{dx^{s+1}} &=
    \sum_{m\in\mathcal{P}_s} \mn \big(f^{(|S_0 (m)|)}(u(x)) D^{S_0
      (m)}u(x)\\
    &+
    \sum_{k=1}^{n}  m_k  (k+1)   f^{(|S_{k}(m)|)}(u(x)) D^{S_{k}(m)} u(x)\big),
  \end{aligned}
\end{equation}
where we point out that in the last expression
  the only terms that actually appear are those for which
  $m_k>0$. Since $m_k-1=(S_{k}(m))_{k}$, by collecting the terms for $m,
  k$ such that $S_k(m)=\widehat m$, \eqref{eq:4} can be written as
\begin{align}\label{eq:5}
  \frac{d^{s+1}f(u(x))}{dx^{s+1}} =
  \sum_{\widehat m\in\mathcal{P}_{s+1}} 
  a_{\widehat m} f^{(|\widehat m|)}(u(x)) D^{\widehat m} u(x),
  \end{align}
  where
  \begin{equation}\label{eq:452}
      a_{\widehat m}=
\begin{cases}
\widetilde{a_{\widehat m}} & \mbox{if } \widehat{m_1}=0 \\
\widetilde{a_{\widehat m}}+\left[
    \begin{array}{c}
      s\\
      S_0^{-1}(\widehat{m})
    \end{array}      
    \right] & \mbox{if } \widehat{m_1}\neq0,
\end{cases}
              \quad
    \widetilde{a_{\widehat m}}=\sum_{\scriptsize\begin{array}{c}\widehat m=S_k (m), \\ k\in\{1,\dots,s\},\\
    m\in\mathcal{P}_{s,k}
  \end{array}
} \mn  m_k  (k+1). 
\end{equation}
For $k\in\{1,\dots,s\}$, and   $ m\in\mathcal{P}_{s,k}$, such that $\widehat m=S_k
(m)$, i.e., $\widehat m_i=m_i$, $i\neq k, k+1$, $\widehat m_{k}=m_{k}-1$,
$\widehat m_{k+1}=m_{k+1}+1$,  we deduce:
\begin{align*}
\mn  m_k  (k+1)&= \frac{s!}{m_1!\dots (m_{k}-1)! m_{k+1}!\dots
  m_s!}(k+1)\\
&=
\frac{s!}{\widehat m_1!\dots \widehat m_{k}! (\widehat m_{k+1}-1)!\dots
  \widehat m_s!}(k+1)\\
&=
\frac{s!}{\widehat m_1!\dots \widehat m_{k}! \widehat m_{k+1}!\dots \widehat m_s!}\widehat
m_{k+1}(k+1).
\end{align*}
Let $\widehat m=S_k (m)$ with  $k<s$, then  one has
$\widehat m_{s+1}=0$.  The only element $m\in\mathcal{P}_{s,s}$ is
$(0,\dots,0,1)\in\mathbb N^s$ and $S_s(m)=(0,\dots,0,1)\in\mathbb N^{s+1}$.
Therefore 
  \begin{align}
    \notag
    \widetilde{a_{\widehat m}}&=\frac{s!}{\widehat m_1!\dots \widehat  m_{s+1}!}\sum_{\scriptsize\begin{array}{c}\widehat m=S_k (m), \\ k\in\{1,\dots,s\},\\
    m\in\mathcal{P}_{s,k}
  \end{array}
} \widehat m_{k+1}(k+1)\\
    \label{eq:450}
    \widetilde{a_{\widehat m}}
&=\frac{s!}{\widehat m_1!\dots \widehat m_{s+1}!}\sum_{k=1}^{s} \widehat
m_{k+1}(k+1)
=\frac{s!}{\widehat m_1!\dots \widehat m_{s+1}!}\sum_{k=2}^{s+1} \widehat m_{k}k.
\end{align}

 On the other hand, if $\widehat{m}_1\neq0$, then:
 \begin{equation}\label{eq:451}
   \left[
    \begin{array}{c}
      s\\
      S_0^{-1}(\widehat{m})
    \end{array}      
    \right]
    =\frac{s!}{(\widehat{m}_1-1)!\widehat{m}_2!\cdots\widehat{m}_s!}
    =\frac{s!}{\widehat{m}_1!\widehat{m}_2!\cdots\widehat{m}_s!
      \widehat{m}_{s+1}!}\widehat{m}_1,
  \end{equation}
  where the last equality holds since, as before, we have
 $\widehat{m}_{s+1}=0$. Then, regardless of $\widehat{m}_1$, \eqref{eq:450} and
 \eqref{eq:451} yield for $\widehat m\in\mathcal{P}_{s+1}$
  \begin{equation}\label{eq:453}
    a_{\widehat m}=\frac{s!}{\widehat m_1!\dots \widehat m_{s+1}!}\sum_{k=1}^{s+1} \widehat
    m_{k}k=\frac{s!}{\widehat m_1!\dots\widehat m_{s+1}!}(s+1)=\mna{s+1}{\widehat m},
  \end{equation}
  since $\widehat m\in\mathcal{P}_{s+1}$ means $\sum_{k=1}^{s+1} \widehat
m_{k}k=s+1$. We deduce from \eqref{eq:5}, \eqref{eq:452} and \eqref{eq:453}
that
\begin{align*}
  \frac{d^{s+1}f(u(x))}{dx^{s+1}} =\sum_{\widehat m\in \mathcal{P}_{s+1}}
\mna{s+1}{\widehat m}
    f^{(|\widehat m|)}(u(x)) D^{\widehat m}u(x),
  \end{align*}
  which concludes the proof by induction.
\end{proof}

\section*{Acknowledgments}
 This research was partially
  supported by Spanish MINECO grants    MTM 2011-22741 and MTM 2014-54388-P.

\end{document}